\documentclass[12pt]{amsart}
\usepackage{amssymb}
\usepackage[all]{xy}

\newtheorem{theorem}{Theorem}[section]
\newtheorem{definition}[theorem]{Definition}
\newtheorem{proposition}[theorem]{Proposition}
\newtheorem{conjecture}[theorem]{Conjecture}

\begin{document}

\title{Complex Hadamard matrices with noncommutative entries}

\author{Teodor Banica}
\address{T.B.: Department of Mathematics, University of Cergy-Pontoise, F-95000 Cergy-Pontoise, France. {\tt teo.banica@gmail.com}}

\subjclass[2000]{46L65 (15B34)}
\keywords{Quantum permutation, Hadamard matrix}

\begin{abstract}
We axiomatize and study the matrices of type $H\in M_N(A)$, having unitary entries, $H_{ij}\in U(A)$, and whose rows and columns are subject to orthogonality type conditions. Here $A$ can be any $C^*$-algebra, for instance $A=\mathbb C$, where we obtain the usual complex Hadamard matrices, or $A=C(X)$, where we obtain the continuous families of complex Hadamard matrices. Our formalism allows the construction of a quantum permutation group $G\subset S_N^+$, whose structure and computation is discussed here.
\end{abstract}

\maketitle

\section*{Introduction}

A complex Hadamard matrix is a square matrix $H\in M_N(\mathbb C)$ whose entries are on the unit circle, $|H_{ij}|=1$, and whose rows are pairwise orthogonal. The basic example of such a matrix is the Fourier one, $F_N=(w^{ij})$ with $w=e^{2\pi i/N}$:
$$F_N=\begin{pmatrix}
1&1&1&\ldots&1\\
1&w&w^2&\ldots&w^{N-1}\\
\ldots&\ldots&\ldots&\ldots&\ldots\\
1&w^{N-1}&w^{2(N-1)}&\ldots&w^{(N-1)^2}
\end{pmatrix}$$

In general, the theory of the complex Hadamard matrices can be regarded as a ``non-standard'' branch of discrete Fourier analysis. For a survey of the potential applications, to quantum physics and quantum information theory questions, see \cite{tz1}.

The Hadamard matrices are also known to parametrize the pairs of orthogonal MASA in the simplest von Neumann algebra, $M_N(\mathbb C)$. This discovery of Popa \cite{pop} has led to deep connections with subfactor theory and some related areas. See \cite{ha1}, \cite{jon}.

One way of understanding the operator algebra significance of the Hadamard matrices is via quantum permutation groups. The general theory here, which goes back to \cite{ba1}, was developed in a number of recent papers, including \cite{ba2}, \cite{bb2}, \cite{bb3}, \cite{bfs}, \cite{bne}, \cite{bi2}, \cite{wa2}.

The idea is very simple. Given an Hadamard matrix $H\in M_N(\mathbb C)$, with rows denoted $H_1,\ldots,H_N\in\mathbb T^N$, the rank one projections $P_{ij}=Proj(H_i/H_j)$ form a magic unitary in the sense of Wang \cite{wa1}, and so produce a representation $C(S_N^+)\to M_N(\mathbb C)$. We can consider then the smallest quantum subgroup $G\subset S_N^+$ producing a factorization of type $C(S_N^+)\to C(G)\to M_N(\mathbb C)$, and call it the ``quantum group associated to $H$''.

As a basic example, for $H=F_N$ we obtain $G=\mathbb Z_N$. Several other things are known about the construction $H\to G$, including the important fact that Jones' planar algebra $P=(P_k)$ associated to $H$ can be recovered from the knowledge of $G$. See \cite{ba1}.

In this paper we develop an extension of this theory. We consider square matrices over a $C^*$-algebra $H\in M_N(A)$, whose entries are unitaries, which commute on all rows and columns, and which are such that the rows and columns are pairwise orthogonal.

In the case $A=\mathbb C$ we obtain the usual Hadamard matrices. More generally, in the commutative case $A=C(X)$ we obtain the continuous families of Hadamard matrices.

In general now, the point with our axioms is that these allow the construction of a magic unitary matrix $P\in M_N(M_N(A))$, via the following formula:
$$(P_{ij})_{ab}=\frac{1}{N}H_{ia}H_{ja}^*H_{jb}H_{ib}^*$$

Thus, associated to our Hadamard matrix $H\in M_N(A)$ is a certain quantum permutation group $G\subset S_N^+$. We will show here how certain basic Hadamard matrix results, mostly related to the correspondence $H\to G$, can be extended to the general case, and we will discuss as well the some features brought by the noncommutative setting.

The paper is organized as follows: 1 is a preliminary section, in 2 we present examples and classification results, in 3 we discuss the general properties of the associated quantum permutation groups, and in 4 we perform some explicit computations.

\medskip

\noindent {\bf Acknowledgements.} I would like to thank K. De Commer for useful discussions.

\section{Definition, basic properties}

Let $A$ be a unital $C^*$-algebra. For most of the applications $A$ will be a commutative algebra, $A=C(X)$ with $X$ being a compact space, or a matrix algebra, $A=M_K(\mathbb C)$ with $K\in\mathbb N$. We will sometimes consider random matrix algebras, $A=M_K(C(X))$, with $X$ being a compact space, and with $K\in\mathbb N$. Our guiding example will be $A=\mathbb C$.

Two row or column vectors over $A$, say $a=(a_1,\ldots,a_N)$ and $b=(b_1,\ldots,b_N)$, will be called orthogonal when $\sum_ia_ib_i^*=\sum_ia_i^*b_i=0$. Observe that, by applying the involution, we have as well $\sum_ib_ia_i^*=\sum_ib_i^*a_i=0$. With this notion in hand, we can formulate:

\begin{definition}
An Hadamard matrix over $A$ is a square matrix $H\in M_N(A)$ such that:
\begin{enumerate}
\item All the entries of $H$ are unitaries, $H_{ij}\in U(A)$.

\item These entries commute on all rows and all columns of $H$.

\item The rows and columns of $H$ are pairwise orthogonal.
\end{enumerate}
\end{definition}

As a first remark, in the simplest case $A=\mathbb C$ the unitary group is the unit circle in the complex plane, $U(\mathbb C)=\mathbb T$, and we obtain the usual complex Hadamard matrices.

In the general commutative case, $A=C(X)$ with $X$ compact space, our Hadamard matrix must be formed of ``fibers'', one for each point $x\in X$. Therefore, we obtain:

\begin{proposition}
The Hadamard matrices $H\in M_N(A)$ over a commutative algebra $A=C(X)$ are exactly the families of complex Hadamard matrices of type
$$H=\{H^x|x\in X\}$$
with $H^x$ depending continuously on the parameter $x\in X$.
\end{proposition}

\begin{proof}
This follows indeed by combining the above two observations. Observe that, when we wrote $A=C(X)$ in the above statement, we used the Gelfand theorem.
\end{proof}

Let us comment now on the above axioms. Since for $U,V\in U(A)$ the commutation relation $UV=VU$ implies as well the commutation relations $UV^*=V^*U$, $U^*V=VU^*$, $U^*V^*=U^*V^*$, the axiom (2) tells us that the $C^*$-algebras $R_1,\ldots,R_N$ and $C_1,\ldots,C_N$ generated by the rows and the columns of $A$ must be all commutative.

We will be particulary interested in the following type of matrices:

\begin{definition}
An Hadamard matrix $H\in M_N(A)$ is called ``non-classical'' if the $C^*$-algebra generated by its coefficients is not commutative.
\end{definition}

Let us comment now on the axiom (3). According to our definition of orthogonality there are 4 sets of relations to be satisfied, namely for any $i\neq k$ we must have:
$$\sum_jH_{ij}H_{kj}^*=\sum_jH_{ij}^*H_{kj}=\sum_jH_{ji}H_{jk}^*=\sum_jH_{ji}^*H_{jk}=0$$

Now since by axiom (1) all the entries $H_{ij}$ are known to be unitaries, we can replace this formula by the following more general equation, valid for any $i,k$:
$$\sum_jH_{ij}H_{kj}^*=\sum_jH_{ij}^*H_{kj}=\sum_jH_{ji}H_{jk}^*=\sum_jH_{ji}^*H_{jk}=N\delta_{ik}$$

The point now is that, in terms of the matrices $H=(H_{ij})$, $H^*=(H_{ji}^*)$, $H^t=(H_{ji})$ and $\bar{H}=(H_{ij}^*)$, these equations simply read:
$$HH^*=H^*H=H^t\bar{H}=\bar{H}H^t=N1_N$$

So, let us recall now that a square matrix $H\in M_N(A)$ is called ``biunitary'' if both $H$ and $H^t$ are unitaries. In the particular case where $A$ is commutative, $A=C(X)$, we have ``$H$ unitary $\implies$ $H^t$ unitary'', so in this case biunitary means of course unitary.

In terms of this notion, we have the following reformulation of Definition 1.1:

\begin{proposition}
Assume that $H\in M_N(A)$ has unitary entries, which commute on all rows and all columns of $H$. Then the following are equivalent:
\begin{enumerate}
\item $H$ is Hadamard.

\item $H/\sqrt{N}$ is biunitary.

\item $HH^*=H^t\bar{H}=N1_N$.
\end{enumerate}
\end{proposition}

\begin{proof}
We know that (1) happens if and only if the axiom (3) in Definition 1.1 is satisfied, and by the above discussion, this axiom (3) is equivalent to (2). Regarding now the equivalence with (3), this follows from the commutation axiom (2) in Definition 1.1.
\end{proof}

Observe that if $H=(H_{ij})$ is Hadamard, so are the matrices $\bar{H}=(H_{ij}^*)$, $H^t=(H_{ji})$ and $H^*=(H_{ji}^*)$. In addition, we have the following result:

\begin{proposition}
The class of Hadamard matrices $H\in M_N(A)$ is stable under the following two operations:
\begin{enumerate}
\item Permuting the rows or columns.

\item Multiplying the rows or columns by central unitaries. 
\end{enumerate}
When successively combining these two operations, we obtain an equivalence relation on the class of Hadamard matrices $H\in M_N(A)$.
\end{proposition}

\begin{proof}
This is clear from definitions. Observe that in the commutative case $A=C(X)$ any unitary is central, so we can multiply the rows or columns by any unitary. In particular in this case we can always ``dephase'' the matrix, i.e. assume that its first row and column consist of $1$ entries (note that this operation is not allowed in the general case).
\end{proof}

Let us discuss now the tensor product operation:

\begin{proposition}
Let $H\in M_N(A)$ and $K\in M_M(A)$ be Hadamard matrices, and assume that $<H_{ij}>$ commutes with $<K_{ab}>$. Then the ``tensor product'' $H\otimes K\in M_{NM}(A)$, given by $(H\otimes K)_{ia,jb}=H_{ij}K_{ab}$, is an Hadamard matrix.
\end{proposition}

\begin{proof}
This follows from definitions, and is as well a consequence of the more general Theorem 1.7 below, that will be proved with full details.
\end{proof}

Following \cite{dit}, the deformed tensor products are constructed as follows:

\begin{theorem}
Let $H\in M_N(A)$ and $K\in M_M(A)$ be Hadamard matrices, and $Q\in M_{N\times M}(A)$ be a rectangular matrix of unitaries, whose entries commute on rows and columns.  Assume in addition that the algebras $<H_{ij}>$, $<K_{ab}>$, $<Q_{ib}>$ pairwise commute. Then the ``deformed tensor product'' $H\otimes_QK\in M_{NM}(A)$, given by
$$(H\otimes_QK)_{ia,jb}=Q_{ib}H_{ij}K_{ab}$$
is an Hadamard matrix as well.
\end{theorem}

\begin{proof}
First, the entries of $L=H\otimes_QK$ are unitaries, and its rows are orthogonal:
\begin{eqnarray*}
\sum_{jb}L_{ia,jb}L_{kc,jb}^*
&=&\sum_{jb}Q_{ib}H_{ij}K_{ab}\cdot Q_{kb}^*K_{cb}^*H_{kj}^*
=N\delta_{ik}\sum_bQ_{ib}K_{ab}\cdot Q_{kb}^*K_{cb}^*\\
&=&N\delta_{ik}\sum_jK_{ab}K_{cb}^*
=NM\cdot\delta_{ik}\delta_{ac}
\end{eqnarray*}

The orthogonality of columns can be checked as follows:
\begin{eqnarray*}
\sum_{ia}L_{ia,jb}L_{ia,kc}^*
&=&\sum_{ia}Q_{ib}H_{ij}K_{ab}\cdot Q_{ic}^*K_{ac}^*H_{ik}^*
=M\delta_{bc}\sum_iQ_{ib}H_{ij}\cdot Q_{ic}^*H_{ik}^*\\
&=&M\delta_{bc}\sum_iH_{ij}H_{ik}^*
=NM\cdot\delta_{jk}\delta_{bc}
\end{eqnarray*}

For the commutation on rows we use in addition the commutation on rows for $Q$:
\begin{eqnarray*}
L_{ia,jb}L_{kc,jb}
&=&Q_{ib}H_{ij}K_{ab}\cdot Q_{kb}H_{kj}K_{cb}
=Q_{ib}Q_{kb}\cdot H_{ij}H_{kj}\cdot K_{ab}K_{cb}\\
&=&Q_{kb}Q_{ib}\cdot H_{kj}H_{ij}\cdot K_{cb}K_{ab}
=Q_{kb}H_{kj}K_{cb}\cdot Q_{ib}H_{ij}K_{ab}
=L_{kc,jb}L_{ia,jb}
\end{eqnarray*}

The commutation on columns is similar, using the commutation on columns for $Q$:
\begin{eqnarray*}
L_{ia,jb}L_{ia,kc}
&=&Q_{ib}H_{ij}K_{ab}\cdot q_{ic}H_{ik}K_{ac}
=Q_{ib}Q_{ic}\cdot H_{ij}H_{ik}\cdot K_{ab}K_{ac}\\
&=&Q_{ic}Q_{ib}\cdot H_{ik}H_{ij}\cdot K_{ac}K_{ab}
=Q_{ic}H_{ik}K_{ac}\cdot Q_{ib}H_{ij}K_{ab}
=L_{ia,kc}L_{ia,jb}
\end{eqnarray*}

Thus all the axioms are satisfied, and $L$ is indeed Hadamard.
\end{proof}

As a basic example, we have the following construction:

\begin{proposition}
The following matrix is Hadamard,
$$M=\begin{pmatrix}x&y&x&y\\ x&-y&x&-y\\ z&t&-z&-t\\ z&-t&-z&t\end{pmatrix}$$
for any unitaries $x,y,z,t$ satisfying $[x,y]=[x,z]=[y,t]=[z,t]=0$.
\end{proposition}

\begin{proof}
This follows indeed from Theorem 1.7, because we have:
$$\begin{pmatrix}1&1\\ 1&-1\end{pmatrix}\otimes_{\begin{pmatrix}x&y\\ z&t\end{pmatrix}}\begin{pmatrix}1&1\\ 1&-1\end{pmatrix}
=\begin{pmatrix}x&y&x&y\\ x&-y&x&-y\\ z&t&-z&-t\\ z&-t&-z&t\end{pmatrix}$$

Observe that $M$ is generically non-classical, in the sense of Definition 1.3.
\end{proof}

\section{Classification results, examples}

The usual complex Hadamard matrices were classified in \cite{ha1} at $N=2,3,4,5$. In this section we investigate the case of the general Hadamard matrices. We use the equivalence relation constructed in Proposition 1.5 above. We first have:

\begin{proposition}
The $2\times 2$ Hadamard matrices are all classical, and are all equivalent to the Fourier matrix $F_2$.
\end{proposition}

\begin{proof}
Consider indeed an arbitrary $2\times 2$ Hadamard matrix:
$$H=\begin{pmatrix}A&B\\ C&D\end{pmatrix}$$

We already know that $A,D$ each commute with $B,C$. Now since we have $AB^*+CD^*=0$, we deduce that $A=-CD^*B$ commutes with $D$, and that $C=-AB^*D$ commutes with $B$. Thus our matrix is classical, any since all unitaries are now central, we can dephase our matrix, which follows therefore to be the Fourier matrix $F_2$.
\end{proof}

Let us discuss now the case $N=3$. Here the classification in the classical case uses the key fact that any formula of type $a+b+c=0$, with $|a|=|b|=|c|=1$, must be, up to a permutation of terms, a ``trivial'' formula of type $a+ja+j^2a=0$, with $j=e^{2\pi i/3}$. 

Here is the noncommutative analogue of this simple fact:

\begin{proposition}
Assume that $a+b+c=0$ is a vanishing sum of unitaries. Then this sum must be of type $a+wa+w^2a=0$, with $w$ unitary satisfying $1+w+w^2=0$.
\end{proposition}

\begin{proof}
Since $-c=a+b$ is unitary we have $(a+b)(a+b)^*=1$, so $ab^*+ba^*=-1$. Thus $ab^*ba^*+(ba^*)^2=-ba^*$, and with $w=ba^*$ we obtain $1+w^2=-w$, and we are done. 
\end{proof}

With this result in hand, we can start the $N=3$ classification. We first have the following technical result, that we will improve later on:

\begin{proposition}
Any $3\times 3$ Hadamard matrix must be of the form
$$H=\begin{pmatrix}a&b&c\\ ua&uv^*w^2vb&uv^*wvc\\ va&wvb&w^2vc\end{pmatrix}$$
with $w$ being subject to the equation $1+w+w^2=0$.
\end{proposition}

\begin{proof}
Consider an arbitrary Hadamard matrix $H\in M_3(A)$. We define $a,b,c,u,v,w$ as for that part of the matrix to be exactly as in the statement, as follows:
$$H=\begin{pmatrix}a&b&c\\ ua&x&y\\ va&wvb&z\end{pmatrix}$$

Let us look first at the scalar product between the first and third row:
$$vaa^*+wvbb^*+zc^*=0$$

By simplifying we obtain $v+wv+zc^*=0$, and now by using Proposition 2.2 we conclude that we have $1+w+w^2=0$, and that $zc^*=w^2v$, and so $z=w^2vc$, as claimed.

The scalar products of the first column with the second and third ones are:
$$a^*b+a^*u^*x+a^*v^*wvb=0$$
$$a^*c+a^*u^*y+a^*v^*w^2vc=0$$

By multiplying to the left by $va$, and to the right by $b^*v^*$ and $c^*v^*$, we obtain:
$$1+vu^*xb^*v^*+w=0$$
$$1+vu^*yc^*v^*+w^2=0$$

Now by using Proposition 2.2 again, we obtain $vu^*xb^*v^*=w^2$ and $vu^*yc^*v^*=w$, and so $x=uv^*w^2vb$ and $y=uv^*wvc$, and we are done.
\end{proof}

We can already deduce now a first classification result, as follows:

\begin{proposition}
There is no Hadamard matrix $H\in M_3(A)$ with self-adjoint entries.
\end{proposition}

\begin{proof}
We use Proposition 2.3. Since the entries are idempotents, we have:
$$a^2=b^2=c^2=u^2=v^2=(uw)^2=(vw)^2=1$$

It follows that our matrix is in fact of the following form:
$$H=\begin{pmatrix}a&b&c\\ ua&uwb&uw^2c\\ va&wvb&w^2vc\end{pmatrix}$$

The commutation between $H_{22},H_{23}$ reads:
\begin{eqnarray*}
[uwb,wvb]=0
&\implies&[uw,wv]=0
\implies uwwv=wvuw\\
&\implies&uvw=vuw^2
\implies w=1
\end{eqnarray*}

Thus we have reached to a contradiction, and we are done.
\end{proof}

Let us go back now to the general case. We have the following technical result, which refines Proposition 2.3 above, and which will be in turn further refined, later on:

\begin{proposition}
Any $3\times 3$ Hadamard matrix must be of the form
$$H=\begin{pmatrix}a&b&c\\ ua&w^2ub&wuc\\ va&wvb&w^2vc\end{pmatrix}$$
where $(a,b,c)$ and $(u,v,w)$ are triples of commuting unitaries, and $1+w+w^2=0$.
\end{proposition}

\begin{proof}
We use Proposition 2.3. With $e=uv^*$, the matrix there becomes:
$$H=\begin{pmatrix}a&b&c\\ eva&ew^2vb&ewvc\\ va&wvb&w^2vc\end{pmatrix}$$

The commutation relation between $H_{22},H_{32}$ reads:
\begin{eqnarray*}
[ew^2vb,wvb]=0
&\implies&[ew^2v,wv]=0
\implies ew^2vwv=wvew^2v\\
&\implies&ew^2v=wvew
\implies [ew,wv]=0
\end{eqnarray*}

Similarly, the commutation between $H_{23},H_{33}$ reads:
\begin{eqnarray*}
[ewvc,w^2vc]=0
&\implies&[ewv,w^2v]=0
\implies ewvw^2v=w^2vewv\\
&\implies&ewv=w^2vew^2
\implies[ew^2,w^2v]=0
\end{eqnarray*}

We can rewrite this latter relation by using the formula $w^2=-1-w$, and then, by further processing it by using the first relation, we obtain:
\begin{eqnarray*}
[e(1+w),(1+w)v]=0
&\implies&[e,wv]+[ew,v]=0\\
&\implies&2ewv-wve-vew=0\\
&\implies&ewv=\frac{1}{2}(wve+vew)
\end{eqnarray*}

We use now the key fact that when an average of two unitaries is unitary, then the three unitaries involved are in fact all equal. This gives:
$$ewv=wve=vew$$

Thus we obtain $[w,e]=[w,v]=0$, so $w,e,v$ commute. Our matrix becomes:
$$H=\begin{pmatrix}a&b&c\\ eva&w^2evb&wevc\\ va&wvb&w^2vc\end{pmatrix}$$

Now by remembering that $u=ev$, this gives the formula in the statement.
\end{proof}

We can now formulate our main classification result, as follows:

\begin{theorem}
The $3\times 3$ Hadamard matrices are all classical, and are all equivalent to the Fourier matrix $F_3$.
\end{theorem}

\begin{proof}
We know from Proposition 2.5 that we can write our matrix in the following way, where $(a,b,c)$ and $(u,v,w)$ pairwise commute, and where $1+w+w^2=0$:
$$H=\begin{pmatrix}a&b&c\\ au&buw&cuw^*\\ av&bvw^*&cvw\end{pmatrix}$$

We also know that $(a,u,v)$, $(b,uw,vw^*)$, $(c,uw^*,vw)$ and $(ab,ac,bc,w)$ have entries which pairwise commute. We first show that $uv$ is central. Indeed, we have:
$$buv
=buvww^*
=b(uw)(vw^*)
=(uw)(vw^*)b
=uvb$$

Similarly, $cuv=uvc$. It follows that we may in fact suppose that $uv$ is a scalar. But since our relations are homogeneous, we may assume in fact that $u=v^*$.

Let us now show that $[abc,vw^*]=0$. Indeed, we have:
\begin{eqnarray*} 
abc
&=&a(bc)ww^*
=aw(bc)w^*
=av(wv^*)bcw^*
=avb(wv^*)cw^*
=v(ab)wv^*cw^*\\
&=&vw(ab)v^*cw^*
=vw(ab)w(w^*v^*)cw^*
=vw^2(ab)c(w^*v^*)w^*
=vw^*abcv^*w
\end{eqnarray*}

We know also that $[b,vw^*]=0$. Hence $[ac,vw^*]=0$. But $[ac,w^*]=0$. Hence $[ac,v]=0$. But $[a,v]=0$. Hence $[c,v]=0$. But $[c,vw]=0$. So $[c,w]=0$. But $[bc,w]=0$. So $[b,w]=0$. But $[b,v^*w]=0$ and $[ab,w]=0$, so respectively $[b,v]=0$ and $[a,w]=0$. Thus all operators $a,b,c,v,w$ pairwise commute, and we are done.
\end{proof}

At $N=4$ now, the classical theory uses the fact that an equation of type $a+b+c+d=0$ with $|a|=|b|=|c|=|d|=1$ must be, up to a permutation of the terms, a ``trivial'' equation of the form $a-a+b-b=0$. In our setting, however, we have for instance:
$$\begin{pmatrix}a&0\\ 0&x\end{pmatrix}+\begin{pmatrix}-a&0\\ 0&y\end{pmatrix}+\begin{pmatrix}b&0\\ 0&-x\end{pmatrix}+\begin{pmatrix}-b&0\\ 0&-y\end{pmatrix}=0$$

It is probably possible to further complicate this kind of identity, and this makes the $N=4$ classification a quite difficult task. We have however the following statement:

\begin{conjecture}
The Hadamard matrices of small size are as follows:
\begin{enumerate}
\item The matrices in Proposition 1.8 are the only ones at $N=4$.

\item There is no $5\times 5$ Hadamard matrix with self-adjoint entries.
\end{enumerate}
\end{conjecture}

We have no computer-assisted evidence for this statement, because all this problematics is quite hard to implement on a computer, but we have instead some encouraging computations at $A=M_2(\mathbb C)$, that we will not detail here, along with some supporting theoretical evidence, coming from \cite{bb1}, \cite{jms}. Indeed, as explained in section 3 below, any Hadamard matrix $H\in M_N(A)$ produces a quantum permutation group $G\subset S_N^+$. The point now is that the quantum subgroups $G\subset S_4^+$ were classified in \cite{bb1}, and the classification of the quantum subgroups $G\subset S_5^+$ is something that can be in principle deduced from \cite{jms}, and looking at the list of the possible candidates gives some evidence for the above conjecture. There is of course some serious work to be done here.

Summarizing, the situation in our generalized Hadamard setting is more complicated than the one in the classical case, fully discussed by Haagerup in \cite{ha1}, up to $N=5$.

There are many interesting classes of special Hadamard matrices, such as the circulant ones \cite{bjo}, \cite{ha2}, the master ones \cite{aff}, and the McNulty-Weigert ones \cite{mwe}, whose theory waits to be extended to the noncommutative setting. Equally interesting is the defect problematics in \cite{tz2}, but an extension here would probably require some geometric conditions, such as assuming that we have a random matrix algebra, $A=M_K(C(X))$.

\section{Quantum permutation groups}

We recall that a magic unitary matrix is a square matrix over a $C^*$-algebra, $u\in M_N(A)$, whose entries are projections ($p^2=p^*=p$), summing up to $1$ on each row and each column. The following key definition is due to Wang \cite{wa1}:

\begin{definition}
$C(S_N^+)$ is the universal $C^*$-algebra generated by the entries of a $N\times N$ magic unitary matrix $u=(u_{ij})$, with the morphisms given by
$$\Delta(u_{ij})=\sum_ku_{ik}\otimes u_{kj}\quad,\quad\varepsilon(u_{ij})=\delta_{ij}\quad,\quad S(u_{ij})=u_{ji}$$
as comultiplication, counit and antipode. 
\end{definition}

This algebra satisfies Woronowicz' axioms in \cite{wo1}, \cite{wo2}, and the underlying noncommutative space $S_N^+$ is therefore a compact quantum group, called quantum permutation group. As a first observation, we have a compact quantum group inclusion $S_N\subset S_N^+$, coming from the representation $\pi:C(S_N^+)\to C(S_N)$ constructed as follows:
$$u_{ij}\to\chi\left(\sigma\in S_N\Big|\sigma(j)=i\right)$$

It is known that this inclusion is an isomorphism at $N=2,3$, but not at $N\geq4$, where $S_N^+$ is non-classical, and infinite. Moreover, it is known that we have $S_4^+\simeq SO_3^{-1}$, and that any $S_N^+$ with $N\geq4$ has the same fusion semiring as $SO_3$. See \cite{bb3}, \cite{wa1}.

The complex Hadamard matrices are known to produce magic matrices, which in turn produce quantum permutation groups. In our generalized setting, we have:

\begin{theorem}
If $H\in M_N(A)$ is Hadamard, the matrices $P_{ij}\in M_N(A)$ defined by
$$(P_{ij})_{ab}=\frac{1}{N}H_{ia}H_{ja}^*H_{jb}H_{ib}^*$$
form altogether a magic matrix $P=(P_{ij})$, over the algebra $M_N(A)$.
\end{theorem}

\begin{proof}
The magic condition can be checked in three steps, as follows:

(1) Let us first check that each $P_{ij}$ is a projection, i.e. that we have $P_{ij}=P_{ij}^*=P_{ij}^2$. Regarding the first condition, namely $P_{ij}=P_{ij}^*$, this simply follows from:
$$(P_{ij})_{ba}^*=\frac{1}{N}(H_{ib}H_{jb}^*H_{ja}H_{ia}^*)^*=\frac{1}{N}H_{ia}H_{ja}^*H_{jb}H_{ib}^*=(P_{ij})_{ab}$$

As for the second condition, $P_{ij}=P_{ij}^2$, this follows from the fact that all the entries $H_{ij}$ are assumed to be unitaries, i.e. follows from axiom (1) in Definition 1.1:
\begin{eqnarray*}
(P_{ij}^2)_{ab}
&=&\sum_c(P_{ij})_{ac}(P_{ij})_{cb}
=\frac{1}{N^2}\sum_cH_{ia}H_{ja}^*H_{jc}H_{ic}^*H_{ic}H_{jc}^*H_{jb}H_{ib}^*\\
&=&\frac{1}{N}H_{ia}H_{ja}^*H_{jb}H_{ib}^*
=(P_{ij})_{ab}
\end{eqnarray*}

(2) Let us check now that fact that the entries of $P$ sum up to 1 on each row. For this purpose we use the equality $H^*H=N1_N$, coming from the axiom (3), which gives:
$$(\sum_jP_{ij})_{ab}
=\frac{1}{N}\sum_jH_{ia}H_{ja}^*H_{jb}H_{ib}^*
=\frac{1}{N}H_{ia}(H^*H)_{ab}H_{ib}^*
=\delta_{ab}H_{ia}H_{ib}^*
=\delta_{ab}$$

(3) Finally, let us check that the entries of $P$ sum up to 1 on each column. This is the trickiest check, because it involves, besides axiom (1) and the formula $H^t\bar{H}=N1_N$ coming from axiom (3), the commutation on the columns of $H$, coming from axiom (2):
\begin{eqnarray*}
(\sum_iP_{ij})_{ab}
&=&\frac{1}{N}\sum_iH_{ia}H_{ja}^*H_{jb}H_{ib}^*
=\frac{1}{N}\sum_iH_{ja}^*H_{ia}H_{ib}^*H_{jb}\\
&=&\frac{1}{N}H_{ja}^*(H^t\bar{H})_{ab}H_{jb}
=\delta_{ab}H_{ja}^*H_{jb}
=\delta_{ab}
\end{eqnarray*}

Thus $P$ is indeed a magic matrix in the above sense, and we are done.
\end{proof}

By using now the general theory of Hopf images from \cite{bb2}, we can formulate:

\begin{definition}
The quantum permutation group $G\subset S_N^+$ associated to an Hadamard matrix $H\in M_N(A)$ is constructed as follows:
\begin{enumerate}
\item We construct a magic matrix $P\in M_N(M_N(A))$ as above.

\item We let $\pi:C(S_N^+)\to M_N(A)$ be the representation associated to $P$.

\item We factorize $\pi:C(S_N^+)\to C(G)\to M_N(A)$, with $G\subset S_N^+$ chosen minimal.
\end{enumerate}
\end{definition}

As an illustration, consider a usual Hadamard matrix $H\in M_N(\mathbb C)$. If we denote its rows by $H_1,\ldots,H_N$ and we consider the vectors $\xi_{ij}=H_i/H_j$, then we have:
$$\xi_{ij}=\left(\frac{H_{i1}}{H_{j1}},\ldots,\frac{H_{iN}}{H_{jN}}\right)$$

Thus the orthogonal projection on this vector $\xi_{ij}$ is given by:
$$(P_{\xi_{ij}})_{ab}=\frac{1}{||\xi_{ij}||^2}(\xi_{ij})_a\overline{(\xi_{ij})_b}=\frac{1}{N}H_{ia}H_{ja}^*H_{jb}H_{ib}^*=(P_{ij})_{ab}$$

We conclude that we have $P_{ij}=P_{\xi_{ij}}$ for any $i,j$, so our construction from Definition 3.3 is compatible with the construction for the usual complex Hadamard matrices.

At the general level, there are many interesting questions, and notably that of understanding how the idempotent state theory from \cite{bfs}, \cite{wa2} can be applied in this setting, in order to obtain an explicit formula for the moments of the main character of $G$.

\section{Free wreath products}

We discuss here the computation of the quantum permutation groups associated to the deformed tensor products of Hadamard matrices. In the classical case this is a key problem, and several results on the subject were obtained in \cite{ba2}, \cite{bb3}, \cite{bi2}.

Let us begin with a study of the associated magic unitary. Regarding the deformed tensor products of Hadamard matrices, we use here the general notations from Theorem 1.7 above. Regarding the quantum algebra part, we use here the various notions introduced in section 3 above, and notably the magic unitary matrix $P=(P_{ij})$ introduced in Theorem 3.2, that we will write now with double indices, $P=(P_{ia,jb})$. We have:

\begin{proposition}
The magic unitary associated to $H\otimes_QK$ is given by
$$P_{ia,jb}=R_{ij}\otimes\frac{1}{N}(Q_{ic}Q_{jc}^*Q_{jd}Q_{id}^*\cdot K_{ac}K_{bc}^*K_{bd}K_{ad}^*)_{cd} $$
where $R_{ij}$ is the magic unitary matrix associated to $H$.
\end{proposition}

\begin{proof}
With the above-mentioned conventions for deformed tensor products and for double indices, the entries of $L=H\otimes_QK$ are by definition the following elements: 
$$L_{ia,jb}=Q_{ib}H_{ij}K_{ab}$$

Thus the projections $P_{ia,jb}$ constructed in Theorem 3.2 are given by:
\begin{eqnarray*}
(P_{ia,jb})_{kc,ld}
&=&\frac{1}{MN}L_{ia,kc}L_{jb,kc}^*L_{jb,ld}L_{ia,ld}^*\\
&=&\frac{1}{MN}(Q_{ic}H_{ik}K_{ac})(Q_{jc}H_{jk}K_{bc})^*(Q_{jd}H_{jl}K_{bd})(Q_{id}H_{il}K_{ad})^*\\
&=&\frac{1}{MN}(Q_{ic}Q_{jc}^*Q_{jd}Q_{id}^*)(H_{ik}H_{jk}^*H_{jl}H_{il}^*)(K_{ac}K_{bc}^*K_{bd}K_{ad}^*)
\end{eqnarray*}

In terms now of the standard matrix units $e_{kl},e_{cd}$, we have:
\begin{eqnarray*}
P_{ia,jb}
&=&\frac{1}{MN}\sum_{kcld}e_{kl}\otimes e_{cd}\otimes(Q_{ic}Q_{jc}^*Q_{jd}Q_{id}^*)(H_{ik}H_{jk}^*H_{jl}H_{il}^*)(K_{ac}K_{bc}^*K_{bd}K_{ad}^*)\\
&=&\frac{1}{MN}\sum_{kcld}\left(e_{kl}\otimes 1\otimes H_{ik}H_{jk}^*H_{jl}H_{il}^*\right)(1\otimes e_{cd}\otimes Q_{ic}Q_{jc}^*Q_{jd}Q_{id}^*\cdot K_{ac}K_{bc}^*K_{bd}K_{ad}^*)
\end{eqnarray*}

Since the quantities on the right commute, this gives the formula in the statement.
\end{proof}

In order to investigate the Di\c t\u a deformations, we use:

\begin{definition}
Let $C(S_M^+)\to A$ and $C(S_N^+)\to B$ be Hopf algebra quotients, with fundamental corepresentations denoted $u,v$. We let
$$A*_wB=A^{*N}*B/<[u_{ab}^{(i)},v_{ij}]=0>$$
with the Hopf algebra structure making $w_{ia,jb}=u_{ab}^{(i)}v_{ij}$ a corepresentation.
\end{definition}

The fact that we have indeed a Hopf algebra follows from the fact that $w$ is magic. In terms of quantum groups, if $A=C(G)$, $B=C(H)$, we write $A*_wB=C(G\wr_*H)$:
$$C(G)*_wC(H)=C(G\wr_*H)$$

The $\wr_*$ operation is the free analogue of $\wr$, the usual wreath product. See \cite{bi1}. 

With this convention, we have the following result:

\begin{theorem}
The representation associated to $L=H\otimes_QK$ factorizes as
$$\xymatrix{C(S_{NM}^+)\ar[rr]^{\pi_L}\ar[rd]&&M_{NM}(\mathbb C)\\&C(S_M^+\wr_*G_H)\ar[ur]&}$$
and so the quantum group associated to $L$ appears as a subgroup $G_L\subset S_M^+\wr_*G_H$.
\end{theorem}

\begin{proof}
We use the formula in Proposition 4.1. For simplifying the writing we agree to use fractions of type $\frac{H_{ia}H_{jb}}{H_{ja}H_{ib}}$ instead of expressions of type $H_{ia}H_{ja}^*H_{jb}H_{ib}^*$, by keeping in mind that the variables are only subject to the commutation relations in Definition 1.1.

Our claim is that the factorization can be indeed constructed, as follows:
$$U_{ab}^{(i)}=\sum_jP_{ia,jb}\quad,\quad V_{ij}=\sum_aP_{ia,jb}$$

Indeed, we have three verifications to be made, as follows:

(1) We must prove that the elements $V_{ij}=\sum_aP_{ia,jb}$ do not depend on $b$, and generate a copy of $C(G_H)$. But if we denote by $(R_{ij})$ the magic for $H$, we have indeed:
$$V_{ij}
=\frac{1}{N}\left(\frac{Q_{ic}Q_{jd}}{Q_{id}Q_{jc}}\cdot\frac{H_{ik}H_{jl}}{H_{il}H_{jk}}\cdot\delta_{cd}\right)_{kc,ld}
=((R_{ij})_{kl}\delta_{cd})_{kc,ld}
=R_{ij}\otimes 1$$

(2) We prove now that for any $i$, the elements $U_{ab}^{(i)}=\sum_jP_{ia,jb}$ form a magic matrix. Since $P=(P_{ia,jb})$ is magic, the elements $U_{ab}^{(i)}=\sum_jP_{ia,jb}$ are self-adjoint, and we have $\sum_bU_{ab}^{(i)}=\sum_{bj}P_{ia,jb}=1$. The fact that each $U_{ab}^{(i)}$ is an idempotent follows from:
\begin{eqnarray*}
((U_{ab}^{(i)})^2)_{kc,ld}
&=&\frac{1}{N^2M^2}\sum_{mejn}\frac{Q_{ic}Q_{je}}{Q_{ie}Q_{jc}}\cdot\frac{H_{ik}H_{jm}}{H_{im}H_{jk}}\cdot\frac{K_{ac}K_{be}}{K_{ae}K_{bc}}\cdot\frac{Q_{ie}Q_{nd}}{Q_{id}Q_{ne}}\cdot\frac{H_{im}H_{nl}}{H_{il}H_{nm}}\cdot\frac{K_{ae}K_{bd}}{K_{ad}K_{be}}\\
&=&\frac{1}{NM^2}\sum_{ejn}\frac{Q_{ic}Q_{je}Q_{nd}}{Q_{jc}Q_{id}Q_{ne}}\cdot\frac{H_{ik}H_{nl}}{H_{jk}H_{il}}\delta_{jn}\cdot\frac{K_{ac}K_{bd}}{K_{bc}K_{ad}}\\
&=&\frac{1}{NM^2}\sum_{ej}\frac{Q_{ic}Q_{je}Q_{jd}}{Q_{jc}Q_{id}Q_{je}}\cdot\frac{H_{ik}H_{jl}}{H_{jk}H_{il}}\cdot\frac{K_{ac}K_{bd}}{K_{bc}K_{ad}}\\
&=&\frac{1}{NM}\sum_j\frac{Q_{ic}Q_{jd}}{Q_{jc}Q_{id}}\cdot\frac{H_{ik}H_{jl}}{H_{jk}H_{il}}\cdot\frac{K_{ac}K_{bd}}{K_{bc}K_{ad}}
=(U_{ab}^{(i)})_{kc,ld}
\end{eqnarray*}

Finally, the condition $\sum_aU_{ab}^{(i)}=1$ can be checked as follows:
$$\sum_aU_{ab}^{(i)}
=\frac{1}{N}\left(\sum_j\frac{Q_{ic}Q_{jd}}{Q_{id}Q_{jc}}\cdot\frac{H_{ik}H_{jl}}{H_{il}H_{jk}}\cdot\delta_{cd}\right)_{kc,ld}
=\frac{1}{N}\left(\sum_j\frac{H_{ik}H_{jl}}{H_{il}H_{jk}}\cdot\delta_{cd}\right)_{kc,ld}=1$$

(3) It remains to prove that we have $U_{ab}^{(i)}V_{ij}=V_{ij}U_{ab}^{(i)}=P_{ia,jb}$. First, we have:
\begin{eqnarray*}
(U_{ab}^{(i)}V_{ij})_{kc,ld}
&=&\frac{1}{N^2M}\sum_{mn}\frac{Q_{ic}Q_{nd}}{Q_{id}Q_{nc}}\cdot\frac{H_{ik}H_{nm}}{H_{im}H_{nk}}\cdot\frac{K_{ac}K_{bd}}{K_{ad}K_{bc}}\cdot\frac{H_{im}H_{jl}}{H_{il}H_{jm}}\\
&=&\frac{1}{NM}\sum_n\frac{Q_{ic}Q_{nd}}{Q_{id}Q_{nc}}\cdot\frac{H_{ik}H_{jl}}{H_{nk}H_{il}}\delta_{nj}\cdot\frac{K_{ac}K_{bd}}{K_{ad}K_{bc}}\\
&=&\frac{1}{NM}\cdot\frac{Q_{ic}Q_{jd}}{Q_{id}Q_{jc}}\cdot\frac{H_{ik}H_{jl}}{H_{jk}H_{il}}\cdot\frac{K_{ac}K_{bd}}{K_{ad}K_{bc}}
=(P_{ia,jb})_{kc,ld}
\end{eqnarray*}

The remaining computation is similar, as follows:
\begin{eqnarray*}
(V_{ij}U_{ab}^{(i)})_{kc,ld}
&=&\frac{1}{N^2M}\sum_{mn}\frac{H_{ik}H_{jm}}{H_{im}H_{jk}}\cdot\frac{Q_{ic}Q_{nd}}{Q_{id}Q_{nc}}\cdot\frac{H_{im}H_{nl}}{H_{il}H_{nm}}\cdot\frac{K_{ac}K_{bd}}{K_{ad}K_{bc}}\\
&=&\frac{1}{NM}\sum_n\frac{Q_{ic}Q_{nd}}{Q_{id}Q_{nc}}\cdot\frac{H_{ik}H_{nl}}{H_{jk}H_{il}}\delta_{jn}\cdot\frac{K_{ac}K_{bd}}{K_{ad}K_{bc}}\\
&=&\frac{1}{NM}\cdot\frac{Q_{ic}Q_{jd}}{Q_{id}Q_{jc}}\cdot\frac{H_{ik}H_{jl}}{H_{jk}H_{il}}\cdot\frac{K_{ac}K_{bd}}{K_{ad}K_{bc}}
=(P_{ia,jb})_{kc,ld}
\end{eqnarray*}

Thus we have checked all the relations, and we are done.
\end{proof}

In general, the problem of further factorizing the above representation is a quite difficult one, even in the classical case. For a number of results here, we refer to \cite{ba2}, \cite{bb3}, \cite{bi2}.


\begin{thebibliography}{99}

\bibitem{aff}J. Avan, T. Fonseca, L. Frappat, P. Kulish, E. Ragoucy and G. Rollet, Temperley-Lieb R-matrices from generalized Hadamard matrices, {\em Theor. Math. Phys.} {\bf 178} (2014), 223--240.

\bibitem{ba1}T. Banica, Hopf algebras and subfactors associated to vertex models, {\em J. Funct. Anal.} {\bf 159} (1998), 243--266.

\bibitem{ba2}T. Banica, Deformed Fourier models with formal parameters, {\em Studia Math.}, to appear.

\bibitem{bb1}T. Banica and J. Bichon, Quantum groups acting on $4$ points, {\em J. Reine Angew. Math.} {\bf 626} (2009), 74--114.

\bibitem{bb2}T. Banica and J. Bichon, Hopf images and inner faithful representations, {\em Glasg. Math. J.} {\bf 52} (2010), 677--703. 

\bibitem{bb3}T. Banica and J. Bichon, Random walk questions for linear quantum groups, {\em Int. Math. Res. Not.} {\bf 24} (2015), 13406--13436.

\bibitem{bfs}T. Banica, U. Franz and A. Skalski, Idempotent states and the inner linearity property, {\em Bull. Pol. Acad. Sci. Math.} {\bf 60} (2012), 123--132.

\bibitem{bne}T. Banica and I. Nechita, Flat matrix models for quantum permutation groups, {\em Adv. Appl. Math.} {\bf 83} (2017), 24--46.

\bibitem{bi1}J. Bichon, Free wreath product by the quantum permutation group, {\em Alg. Rep. Theory} {\bf 7} (2004), 343--362.

\bibitem{bi2}J. Bichon,  Quotients and Hopf images of a smash coproduct, {\em Tsukuba J. Math.} {\bf 39} (2015), 285--310.

\bibitem{bjo}G. Bj\"orck, Functions of modulus $1$ on ${\rm Z}_n$ whose Fourier transforms have constant modulus, and cyclic $n$-roots, {\em NATO Adv. Sci. Inst. Ser. C Math. Phys. Sci.} {\bf 315} (1990), 131--140.

\bibitem{dit}P. Di\c t\u a, Some results on the parametrization of complex Hadamard matrices, {\em J. Phys. A} {\bf 37} (2004), 5355--5374.

\bibitem{ha1}U. Haagerup, Orthogonal maximal abelian $*$-subalgebras of the $n\times n$ matrices and cyclic $n$-roots, in ``Operator algebras and quantum field theory'', International Press (1997), 296--323.

\bibitem{ha2}U. Haagerup, Cyclic $p$-roots of prime lengths $p$ and related complex Hadamard matrices, preprint 2008.

\bibitem{jon}V.F.R. Jones, Planar algebras I, preprint 1999.

\bibitem{jms}V.F.R. Jones, S. Morrison and N. Snyder, The classification of subfactors of index at most 5, {\em Bull. Amer. Math. Soc.} {\bf 51} (2014),  277--327.

\bibitem{mwe}D. McNulty and S. Weigert, Isolated Hadamard matrices from mutually unbiased product bases, {\em J. Math. Phys.} {\bf 53} (2012), 1--21.

\bibitem{pop}S. Popa, Orthogonal pairs of $*$-subalgebras in finite von Neumann algebras, {\em J. Operator Theory} {\bf 9} (1983), 253--268.

\bibitem{tz1}W. Tadej and K. \.Zyczkowski, A concise guide to complex Hadamard matrices, {\em Open Syst. Inf. Dyn.} {\bf 13} (2006), 133--177.

\bibitem{tz2}W. Tadej and K. \.Zyczkowski, Defect of a unitary matrix, {\em Linear Algebra Appl.} {\bf 429} (2008), 447--481.

\bibitem{wa1}S. Wang, Quantum symmetry groups of finite spaces, {\em Comm. Math. Phys.} {\bf 195} (1998), 195--211.

\bibitem{wa2}S. Wang, $L_p$-improving convolution operators on finite quantum groups, {\em Indiana Univ. Math. J.} {\bf 65} (2016), 1609--1637.

\bibitem{wo1}S.L. Woronowicz, Compact matrix pseudogroups, {\em Comm. Math. Phys.} {\bf 111} (1987), 613--665.

\bibitem{wo2}S.L. Woronowicz, Tannaka-Krein duality for compact matrix pseudogroups. Twisted SU(N) groups, {\em Invent. Math.} {\bf 93} (1988), 35--76.

\end{thebibliography}
\end{document}